\documentclass[12pt]{article}

\usepackage{amsmath,amsthm}
\usepackage{amssymb}
\usepackage{csquotes}
\usepackage{newtxtext,newtxmath}
\usepackage{fancyhdr}
\usepackage{authblk}
\usepackage{titlesec}
\usepackage{geometry}
\usepackage{setspace}
\usepackage{doi}
\titlelabel{\thetitle.\quad}

\newtheorem{theorem}{Theorem}[section]

\newtheorem{lemma}[theorem]{Lemma}

\theoremstyle{definition}

\numberwithin{equation}{section}

\begin{document}

\title{Intersective Polynomials Arising from Sums of Powers}

\author{Bhawesh Mishra \footnote{Department of Mathematics, The Ohio State University, MW 242, 231 W. 18th Ave., Columbus, OH 43210 }}

\maketitle



\begin{abstract}
Given a natural number $n \geq 2$, an integer $k$ and for a judiciously chosen $l = l(n)$ we give necessary and sufficient conditions for the polynomial $f_{n,k} = \big( \sum_{i=1}^{l} x_{i}^{n} \big) - k$ to have roots modulo every positive integer. 
\end{abstract}

\maketitle

\section{Introduction.} A set  $S \subset\mathbb{Z}$ is said to be intersective if given any set T $\subset\mathbb{Z}$ with positive upper density, one has\\ $ S$ $\cap$ $(T - T) \not\subseteq \{0\}$. The upper density of $T \subset\mathbb{Z}$ is defined as $\overline{d}(T) := \limsup_{n\rightarrow\infty} \frac{|T \cap \{-n,... , -2,-1, 0, 1, 2, ... , n\}|}{2n+1}$. A polynomial $f(x_{1}, ... , x_{n}) \in\mathbb{Z} [x_{1}, ... , x_{n}]$ is called intersective if and only if the set of its values $\{ f(x_{1}, ... , x_{n}) : x_{1}, x_{2}, ... , x_{n}\in\mathbb{Z} \}$ is intersective.

As the definition of an intersective set suggests, intersective polynomials are of interest in additive combinatorics. For example, S\'ark\"ozy and Furstenberg independently proved, in \cite{Sa1} and \cite{Fu} respectively, that the polynomial $f(x) = x^{2}$ is intersective. In other words, they showed that for every $S \subset\mathbb{N}$ with $\overline{d}(S) > 0$, $(S-S)$ contains a perfect square. The methods in \cite{Sa1} can be modified to show that for any $i \in\mathbb{N}$, the polynomial $f(x) = x^{i}$ is intersective \cite{Sa2}. Kamae and Mend\'es-France showed that a polynomial $ f(x) \in\mathbb{Z}[x]$ is intersective if and only if for each positive integer $m$, $f(x) \equiv 0$ ( mod $m$) is solvable \cite{KaMF}. A very special case of the Polynomial Szemer\'edi's Theorem along intersective polynomials, obtained by Bergelson, Leibman and Lesigne in \cite{BLL}, implies that a polynomial $g(x_{1}, ... , x_{m}) \in\mathbb{Z}[x_{1}, ... , x_{m}]$ in many variables is intersective if and only if for every positive integer $k$ there exists $n_{1}, n_{2}, ... , n_{m} \in\mathbb{Z}$ such that $g(n_{1}, n_{2}, ... , n_{m}) \equiv 0 $ ( mod $k)$.

D. Berend and Y. Bilu, in \cite{BerBil}, obtained a necessary and sufficient condition for any polynomial $f(x)$ of one variable to be intersective. Their result also implies that any polynomial in one variable that is intersective, but has no rational roots, has to be of degree greater than 4. However, there are single-variable intersective polynomials, of degree 5 and greater, that have no rational roots. For example, $h(x) = (x^{3} - 19) (x^{2} + 3) $ is intersective but has no rational roots \cite{HLS}. One can also show that if $p, q$ are distinct odd primes such that $p \equiv q \equiv 1$ ( mod $4)$ and $p$ is a square modulo $q$, then the polynomial $f(x) = (x^{2} - p) (x^{2} - q) (x^{2} - pq)$ is intersective.

It is well known that for any $n \geq 2$ and any $k \in\mathbb{Z}$, $(x^{n} - k)$ is intersective if and only if $k$ is a $n^{th}$ power. One can easily show that for any $k \in\mathbb{Z}$, the polynomial $(x^{2} + y^{2} - k)$ is intersective if and only if $k$ is sum of two integer squares. It can also be shown that $(x^{2} + y^{2} + z^{2} - k)$ is intersective if and only if $k$ is a sum of three integer squares. As a consequence of the Lagrange's four square theorem, we have that for every $l \geq 4$ and for every $k \in\mathbb{Z}$ the polynomial $\big( \sum_{i=1}^{l} x_{i}^{2} \big) - k $ is intersective. Similarly, one can ask about necessary and sufficient conditions for intersectivity polynomials that arise analogously from sums of cubes, fourth-powers and so on. Given $n \geq 2$ and $k \in\mathbb{Z}$ we obtain the necessary and sufficient conditions for the polynomial of several variables $f_{n,k}(x_{1}, ... , x_{l})$ = $\big( \sum_{i=1}^{l} x_{i}^{n} \big) - k $ to be intersective. Here the number of variables $l$ in $f_{n,k}$ will depend upon $n$. 

Polynomials of the form $\big( \sum_{i=1}^{l} x_{i}^{n} \big) - k $ are also studied in the context of the Waring's problem. The Waring's problem states that given $n \geq 3$, there exists a $r = r(n)$ such that the equation $\big( \sum_{i=1}^{r} x_{i}^{n} \big) - k = 0 $ is solvable over integers for every $k \in\mathbb{Z}$. In other words, the Waring's problem states that for each $n \geq 3$ there is a $r = r(n)$ such that every integer is a sum of at most $r$ $n^{th}$ powers. The Waring's problem has a positive answer; the values of $r = r(n)$ are known too \cite{Vaughan}. Another variant of the Waring's problem deals with determining $r^{\prime} = r^{\prime}(n)$ such that every sufficiently large number $k$ is sum of at most $r^{\prime}$ $n^{th}$ powers. If $L = L(n)$ is the smallest positive integer such that $\big( \sum_{i=1}^{L} x_{i}^{n} \big) - k$ is intersective for every $k \in\mathbb{Z}$ then $r^{\prime}(n) \geq L(n)$ \cite{Vaughan}. Studying intersectivity of polynomials of the form $\big( \sum_{i=1}^{l} x_{i}^{n} \big) - k$ can be useful in determining $r^{\prime}(n)$ in the Waring's problem (see \cite{Vaughan} for a comprehensive survey on the Waring's problem).

We shall see that if we take a large enough $l = l(n)$, the polynomial $\big( \sum_{i=1}^{l} x_{i}^{n} \big) - k $ is intersective for every $k \in\mathbb{Z}$. On the other hand, if $l = l(n)$ is too small then the necessary and sufficient conditions for $f_{n,k}(x_{1}, ... , x_{l})$ = $\big( \sum_{i=1}^{l} x_{i}^{n} \big) - k $ to be intersective might be too complicated for practical use. Therefore a bargain between the number of variables in the equation and the desired necessary and sufficient conditions for intersectivity must be made. When $n$ is odd we take $l = \lceil \frac{n+1}{2} \rceil$ and when $n$ is even we take $l$ = max $\{ \lceil \frac{2n}{3} \rceil , \lceil \frac{n+2}{2} \rceil \}$. 

The notations used in this article are explained in the next section. Section $3$ contains some elementary number-theoretic lemmas and Section $4$ consists of relevant results from additive number theory. The necessary and sufficient conditions for intersectivity of  $\big( \sum_{i=1}^{l} x_{i}^{n} \big) - k $, for odd and even $n$ are proved in Sections $5$ and $6$ respectively. Section $7$ presents some illustrative examples of the results obtained in Sections $5$ and $6$.

\section{Notations}
\begin{enumerate}
\item Given an odd natural number $n \geq 3$ and $k\in\mathbb{Z}$, we define $ l(n) = l = \lceil \frac{n+1}{2} \rceil$ and \\ $f_{n,k} := f_{n,k}(x_{1}, ... , x_{l}) =  \big[ \big( \sum_{i=1}^{l} x_{i}^{n} \big) - k \big] $. 

\item Given an even $n \in\mathbb{N}$, $k \in\mathbb{Z}$, we define $l(n) = l = $ max $\{ \lceil \frac{2n}{3} \rceil, \lceil \frac{n+2}{2} \rceil \}$ and \\ $f_{n,k} := f_{n,k}(x_{1}, ... , x_{l}) =  \big[ \big( \sum_{i=1}^{l} x_{i}^{n} \big) - k \big] $.

\item Given $n \in\mathbb{N}$, $k \in\mathbb{Z}$ we define $g_{n,k}:= g_{n,k} (x_{1}, ... , x_{l}) = \big( \sum_{i=1}^{l} x_{i}^{n} \big) - k$, where $l \in\mathbb{N}$. The difference between $g_{n,k}$ and $f_{n,k}$ is that the number of variables $l$ is fixed in $f_{n,k}$ as $l(n)$ but not in $g_{n,k}$. 

\item Given a prime $p$ we define the range of diagonal form g$_{n,0}$(x$_{1}$, ... , x$_{n}$) in $\mathbb{Z}/p\mathbb{Z}$ as \\ $R_{p}(g_{n,0}) := \{ g_{n,0}(y_{1}, ... , y_{n}) : y_{1}, ... , y_{n} \in\mathbb{Z}/p\mathbb{Z}  \}$. 

\item Given a prime $p$ and a diagonal form g$_{n,0}$(x$_{1}$, ... , x$_{n}$) we define\\ $R_{p}^{*}(g_{n,0}):= \{ g_{n,0}(y_{1}, ... , y_{n}) :$ $\exists$ $1 \leq i \leq n $ with $y_{i} \in(\mathbb{Z}/p\mathbb{Z})^{*} =  (\mathbb{Z}/p\mathbb{Z})\setminus\{0\} \}$.  

\item Given a prime $p$, $A_{d}$ will be used to denote the set of $d$-th power residues in ($\mathbb{Z}/p\mathbb{Z}$). We will denote the set $A_{d}\setminus\{0\}$, i.e. the set of non-zero $d$-th power residues, by $A_{d}^{*}$. 

\end{enumerate}

\section{Some Preliminary Lemmas}

To ensure that the polynomial $f_{n,k}$ has roots modulo every integer $n > 1$ , it is enough to find roots of $f_{n,k}$ modulo $q^{j}$ for all primes $q$ and  $j \geq 1$. This follows from the Chinese Remainder Theorem. For most of the primes $q$, one only needs to find non-zero roots of $f_{n,k}$ modulo certain low powers of $q$ to ensure that $f_{n,k}$ has roots modulo all powers of $q$. This fact is formalized in Lemma $3.1$ and Theorem $3.3$ below. 

\begin{lemma}{\textbf{[Hensel's Lemma]}}
Let $h(x_{1}, ... , x_{n})$ be a polynomial with integer coefficients in $n \geq 1$ variables such that $h (y_{1}, y_{2}, ... , y_{n}) \equiv 0 \hspace{1mm} (\text{mod} \hspace{1mm} p^{j})$ for some j $\geq 1$ and y$_{1}$, ... , y$_{n}$ $\in\mathbb{Z}$. If\\ p $\nmid$ $\frac{\partial h }{\partial x_{i}}(y_{1}, ... , y_{n})$ for some 1 $\leq$ i $\leq$ n then $\exists$ z$_{1}$, ... , z$_{n}$ $\in\mathbb{Z}$ such that $\forall$ $1 \leq i \leq n$, $z_{i} \equiv y_{i} \hspace{1mm} (\text{mod}\hspace{1mm} p^{j})$ and $h (z_{1}, z_{2}, ... , z_{n}) \equiv 0 \hspace{1mm} (\text{mod} \hspace{1mm} p^{j+1})$. [See \cite{Rosen}, page 155 for a proof ]
\end{lemma}

If $q$ is a prime such that $q \nmid n$ then any non-zero solution of $h (y_{1}, y_{2}, ... , y_{n}) \equiv 0 \hspace{1mm} ( \text{ mod} \hspace{1mm} q)$ can be lifted to the solution of h(y$_{1}$, ... , y$_{n}$) $\equiv$ 0 ( mod $q^{i}$) for every $i \geq 1$, through an inductive application of Hensel's Lemma. However, if $q \mid n$ then Hensel's Lemma does not allow for lifting of the solution of $f_{n,k} \equiv$ 0 ( mod $q$) to solutions of $f_{n,k} \equiv 0$ ( mod $q^{i}$) , $i \geq 2$. For such primes $q$, we need a different result that can allow us to lift the root of $f_{n,k}$, modulo lower powers of $q$, to that modulo higher powers of $q$. 

To proceed further, we need an elementary lemma which is stated and proved below. Given a prime $p$, $a,n\in\mathbb{N}$ we denote by $p^{a} \mid\mid n$ the statement that $p^{a}$ is the highest power of $p$ dividing $n$. 

\begin{lemma}
Let $p$ be a prime, $n, i \in\mathbb{N}$ with $1 \leq i \leq n$. Let $p^{a} \mid\mid n $ for some $a \geq 1$  and $p^{b} \mid\mid i$ , for some $b \leq a$. Then $p^{a-b}$ $\mid\mid$ $\binom{n}{i}$. 
\end{lemma}

\begin{proof}
Since $p^{a} \mid\mid n$ and $p^{a} \mid\mid i$, assume $n = p^{a}m$  and  $i = p^{b}k$ , where $m, k \in\mathbb{N}$ such that $p \nmid m$ and $p \nmid k$. 

Note that $\binom{n}{i}$ = $\frac{n (n-1) (n-2) ... (n-i+1)}{i (i-1) (i-2) ... 1}$ and that the multiples of $p$ that appear in denominator are $p^{b}k$, ($p^{b}k - p$ ), ($p^{b}k - 2p$) , ... , ($p^{b}k$ - ($p^{b-2}k)p$). Similarly, the multiples of p that appear in the numerator are $p^{a}m$, ($p^{a}m-p$) , ($p^{a}m - 2p$) , ... , ($p^{a}m$ - ($p^{b-2}k)p$). 

Both the numerator and the denominator have a total of $(p^{b-2}k + 1)$ terms that are multiple of $p$.  Apart from the first factors in the numerator and the denominator, i.e. $p^{a}m$ and $p^{b}k$ respectively, the power of $p$ factored from the factors of the form ($p^{a}m$ - $lp$)  in the numerator is canceled by the power of $p$ coming from the terms of the form ($p^{b}k$ - $lp$) in the denominator.  for any $l \geq 1$. This exactly means that $p^{a-b} \mid\mid$ $\binom{n}{i}$.
\end{proof}

\begin{theorem}
Let $p$ be a prime,  $n \in\mathbb{N}$, $k\in\mathbb{Z}$ and $p^{a} \mid\mid $ n . Then  f$_{n,k}$(x$_{1}$, ... , x$_{l}$) $\equiv$ 0 ( mod $p^{i}$) has solutions for every  $i \geq$ 1 if and only if $(1)$ or $(2)$ is satisfied:

\begin{enumerate}

\item $ \exists$ $(x_{i})_{i=1}^{l}\in\mathbb{Z}^{l}$ such that $f_{n,k} (x_{1}, ... , x_{l}) = 0 $ i.e. $k$ is sum of $l$ number of integer $n$-th powers. 

\item
    \begin{itemize}

        \item If p is an odd prime and n is odd, then f$_{n,k}$(x$_{1}$, ... , x$_{l}$) $\equiv$ 0 ( mod $p^{j}$) has solutions for some $j \geq (a+1)$.
    
        \item If p is an odd prime and n is even, then $f_{n,k}(x_{1}, ... , x_{l})$ $\equiv$ 0 ( mod $p^{j}$) has solution  $(x_{1}, ... , x_{l})$, for some $j\geq (a+1)$, such that $p \nmid x_{i}$ for some $1 \leq i \leq l$.
    
        \item If p = 2 then $f_{n,k} (x_{1}, ... , x_{l})$ $\equiv 0$ ( mod $2^{j}$) has solution $(x_{1}, ... ,x_{l})$, for some $j \geq(a+2)$ such that $2 \nmid x_{i}$ for some $ 1 \leq i \leq l$.

\end{itemize}
\end{enumerate} 

\end{theorem}

\begin{proof} \textbf{Proof of Necessity} Assume that (1) is not satisfied  i.e. $k$ is not of the form $(x_{1}^{n} + ... + x_{l}^{n})$. We want to show that condition (2) holds. If $n$ is odd and $p$ is an odd prime, then the result follows from the definition of intersectivity of a polynomial. 
        
Now assume that $p$ is odd and $n$ is even. Since $f_{n,k}(x_{1}, ... , x_{l})$ is solvable modulo all powers of $p$, it follows that $f_{n,k}(x_{1}, ... , x_{l})$ is solvable modulo $p^{j}$ for any $j \geq (a+1)$. Assume that for any $j \geq (a+1)$, if $f_{n,k}(x_{1}, ... , x_{l})$ $\equiv$ 0 ( mod $p^{j}$) then $p \mid x_{i}$ for all $ 1 \leq i \leq l $.  

If $p^{b}\mid\mid k$ with $ 0\leq b < p^{a}$, then choose $x_{1}, ... , x_{l}$ such that $f_{n,k}(x_{1}, ... , x_{l})$ $\equiv$ 0 ( mod $p^{c}$) for $c =$ max $\{a+1, b+1\}$. Since $p^{b}\mid\mid k$, $p^{c} \nmid k$ and $f_{n,k}(x_{1}, ... , x_{l})$ $\equiv$ 0 ( mod $p^{c}$) we have that $p^{c} \nmid (x_{1}^{n} + ... + x_{l}^{n})$ which is contradiction because $p^{p^{a}} \mid (x_{1}^{n} + ... + x_{l}^{n})$ and $c\leq p^{a}$. 

If $p^{b} \mid\mid k$ with $b \geq p^{a}$ then choose a solution $(x_{1}, ... , x_{l})$ to $f_{n,k}(x_{1}, ... ,x_{l})  \equiv 0$ ( mod $p^{j_{0}}$) for $j_{0} = (\alpha p^{a} + a + 1 )$, where $\alpha\in\mathbb{N}$ is chosen such that $p^{\alpha p^{a}} \mid\mid k$. 

Since $p^{p^{a}} \mid (x_{1}^{n} + ... + x_{l}^{n})$, one can divide $ (x_{1}^{n} + ... + x_{l}^{n}) - k \equiv$ 0 ( mod $p^{j}$) on both sides by $p^{p^{a}}$. This gives us $(y_{1}^{n} + ... + y_{l}^{n}) - k^{\prime} \equiv 0$ ( mod $p^{j_{1}}$). Here $y_{i} = \frac{x_{i}}{p^{p^{a}}}$, $k^{\prime} = \frac{k}{p^{p^{a}}}$ and $j_{1} = \frac{j_{0}}{p^{p^{a}}}$. 

Finally we obtain $ (z_{1}^{n} + ... + z_{l}^{n}) - k^{\prime\prime} \equiv$ 0 ( mod $p^{a+1}$) for some $z_{1}, ... , z_{l} \in\mathbb{Z}$ and $k^{\prime\prime} = \frac{k}{p^{\alpha p^{a}}}$, after repeating this process $\alpha$ number of times. 
This lands us back in the previous case where $p^{b} \mid\mid k$ with $0 \leq b < p^{a}$ because $k^{\prime\prime}$ is not divisible by $p^{p^{a}}$ anymore. Therefore, we are done. 

The case when $n$ is even and $p = 2$ is analogous, except we replace all the $(a+1)$ by $(a+2)$.\\

\textbf{Proof of Sufficiency} Assume that $ f_{n,k}(x_{1}, ... , x_{l}) \equiv $ 0 ( mod $p^{j}$) is solvable where $ j \geq (a+1)$ if $p$ is an odd prime and $j \geq (a+2)$ if $p = 2$. Since $ f_{n,k}(x_{1}, ... , x_{l}) = \big( \sum_{i=1}^{l} x_{i}^{n} \big) - k $ , we have that $\big( \sum_{i=1}^{l} x_{i}^{n} \big) - k \equiv 0 $ ( mod $p^{j}) $ i.e. for some $\alpha\in\mathbb{N}$ with $p\nmid\alpha$  we have \begin{equation}
    \big( \sum_{i=1}^{l} x_{i}^{n} \big) - k = \alpha p^{j}.  
\end{equation}
 
Note that if $n$ is odd and $p  \mid x_{i}$ for every $1 \leq i \leq l $, then $ p^{a+1} \mid k $. Then $ (1, p^{a+1} - 1, 0, ... , 0 ) $ is be a solution to $(3.1)$ with $ p \nmid 1$. This along with our hypothesis implies that without loss of generality, $p \nmid x_{1}$.

Now choose $c \in\mathbb{Z}$ such that $c \equiv \frac{-\alpha}{c_{0}x_{1}^{n-1}}$ ( mod $p$), where $c_{0} = \frac{n}{p^{a}}$. Note that such a $c$ exists because $p \nmid x_{1}$ and $p \nmid c_{0}$. Note that $c \not\equiv 0$ ( mod $p$) because $p \nmid \alpha$. 

Define $ y_{1} = (x_{1} + c.p^{j-a})$, so $y_{1}^{n} = \sum_{r=0}^{n} \binom{n}{r} x_{1}^{n-r}.c^{r}.p^{rj-ra} $. Now assume that $p^{b} \mid\mid$ r , then we have $ b \leq log_{p}(r)$. From Lemma $3.2$, we also have

\begin{equation}
p^{rj-ra+a-b} \mid\mid \Big[ \binom{n}{r}x_{1}^{n-r}c^{r}p^{rj-ra} \Big].
\end{equation} 
Now we will deal with the rest of the proof depending upon whether $p$ is odd or $p = 2$. 

\begin{itemize}

\item \textbf{ p is an odd prime}

Since $j \geq (a+1)$ we have that 
\begin{align*}
&rj-ra+a-b = (r-1).(j-a)+j-b  \\ &\geq r-1+j -b  \geq r-1+j-log_{p}(r) 
\\&\geq   j + ( r - 1 - log_{p}(r) )  \geq j + ( r - 1 - \frac{ln(r)}{ln(p)} )  \\&= j + h(r) , 
\end{align*}
where $ h(r) =  r - 1 - \frac{ln(r)}{ln(p)} $ is a function that is increasing with $r$ and $h(2) = 1 - \frac{ln(2)}{ln(p)} > 0 $. This gives us that for every $ r \geq 2$ , $ h(r) \geq 1 $ and hence $rj-ra+a-b \geq (j+1)$. 

Therefore it follows from $(3.2)$ that for every $ r \geq 2, p^{j+1} \mid \Big[ \binom{n}{r}x_{1}^{n-r}c^{r}p^{rj-ra} \Big]$. So we have, 
\begin{align*}
&y_{1}^{n} = \sum_{r=0}^{n} \binom{n}{r} x_{1}^{n-r}c^{r}p^{rj-ra} \equiv  x_{1} + \binom{n}{1}x_{1}^{n-1} c  p^{j-a} \\ &\equiv x_{1}^{n} + c \big( c_{0}x_{1}^{n-1}p^{j} \big) (\text{ mod }  p^{j+1}). \end{align*}
Therefore, 
\begin{align*}
&y_{1}^{n} + x_{2}^{n} + ... + x_{l}^{n} - k \equiv \big[ x_{1}^{n} + x_{2}^{n} + ... + x_{l}^{n} - k \big] + c \big( c_{0}x_{1}^{n-1}p^{j} \big)\\ &\equiv p^{j} \big( \alpha + c c_{0}x_{1}^{n-1} \big)  \equiv 0  (\text{ mod } p^{j+1}) 
\end{align*}
Here the second equivalence uses $(3.1)$ and the last equivalence is due to the choice of c. So $(y_{1}, x_{2}, ... , x_{l})$ is a solution to $f_{n,k} (x_{1}, ... , x_{l} ) \equiv 0 $( mod p$^{j+1})$. 

Since $ p \nmid x_{1}$, $p \nmid y_{1}$ . Therefore, we can repeat the whole process inductively to obtain solution to $f_{n,k} (x_{1}, ... , x_{l} ) \equiv 0 $( mod p$^{i}$) for any $ i \geq j $. Hence we have shown the sufficiency of the above condition for odd primes p. 

\item \textbf{p = 2}

Since $j \geq (a+2) $ we have that 
\begin{align*} &rj-ra+a-b = (r-1)(j-a-1) + r + j - 1 - b   \\ &\geq (r-1) + (r + j - 1 - b) \geq j + (2r - 2 - log_{2}(r) ) \\ & = j + g(r), \end{align*} where $g(r) = (2r - 2 - log_{2}(r) )$ is an increasing function of r and g(2) = 1. This implies that for any $ r \geq 2 $, we have $ rj-ra+a-b \geq  (j+1) $. Therefore, $(3.2)$ implies that for any $r \geq 2$ $2^{j+1} \mid  \Big[ \binom{n}{r}x_{1}^{n-r}c^{r}p^{rj-ra} \Big]$. Therefore we have

\begin{align*}y_{1}^{n} \equiv x_{1}^{n} + nx_{1}^{n-1}c2^{j-a}  \equiv x_{1}^{n} + c_{0}x_{1}^{n-1}c2^{j} (\text{ mod } 2^{j+1}). \end{align*}. 

Hence we have,  
\begin{align*}
&(y_{1}^{n} + x_{2}^{n} + ... + x_{l}^{n} - k ) \equiv \big(  x_{1}^{n} + x_{2}^{n} + ... + x_{l}^{n} - k \big) + c_{0}x_{1}^{n-1}c2^{j} \\ &\equiv 2^{j} \big( \alpha + c.x_{1}^{n-1}c_{0} \big)  \equiv 0 ( \text{ mod } 2^{j+1}). \end{align*} where the second equivalence is from $(3.1)$ and last equivalence is due to the choice of c. This shows that $(y_{1}, x_{2}, ... , x_{l})$ is a solution to $f_{n,k} (x_{1}, ... , x_{l} ) \equiv 0  $( mod 2$^{j+1})$. 

Since $ 2 \nmid x_{1}$, $2 \nmid y_{1}$. Therefore, we can repeat the whole process inductively to obtain solution to $f_{n,k} (x_{1}, ... , x_{l} ) \equiv 0 $( mod 2$^{i}$) for every $ i \geq j $. Hence we have shown the sufficiency of the above condition for prime p = 2. 
\end{itemize}
\end{proof}

\begin{lemma}
Let $p$ be an odd prime of the form $p = (ds+1) $ for some $ d,s \in\mathbb{N}$ and  define $A_{d} = \{ x^{d} : x\in\mathbb{Z}/p\mathbb{Z} \}$. Then $|A_{d}| = (s+1) $. 
\end{lemma}

\begin{proof}
Since $p$ is an odd prime we know that the set of units in $\mathbb{Z}/p\mathbb{Z}$ is generated by a single element i.e. ($\mathbb{Z}/p\mathbb{Z})^{*}$ = $\langle a \rangle$. Note that $(a^{i})^{d} \equiv (a^{j})^{d}) $ ( mod $p$) if and only if $(p-1) \mid d(i-j) $ if and only if $ds \mid d(i-j)$ if and only if $s \mid (i-j)$. This implies that there are exactly $s$ number of $d$-th powers in ($\mathbb{Z}/p\mathbb{Z})^{*}$ , $a^{0}, a^{1}, ... , a^{s-1}$ , i.e. we have $|A_{d}| = (s+1) $. 

\end{proof}
 
\section{Results from Additive Number Theory}

In this section, we collect some results regarding sumsets in finite fields $\mathbb{Z}/p\mathbb{Z}$, which we will repeatedly use in the later sections.  

\begin{lemma}
Let  $n \geq 1$ and $p$ be a prime number. Let A$_{n}$ = $\{ x^{n} : x \in\mathbb{Z}/p\mathbb{Z} \}$. If $d = (n, p-1)$, then A$_{n}$ = A$_{d}$. 
\end{lemma}
See \cite{Melvyn}, page 60. 

\begin{theorem}
Let $ p > 3 $ be a prime number and  $n \in\mathbb{N}$ such that  $1 < (n, p-1) < \frac{p-1}{2}$. Recall that $g_{n,0}(x_{1}, ... , x_{l}) = \sum_{i=1}^{l} x_{i}^{n}$ for some  $l \geq$ 1 . Then the cardinality of the range of $g_{n,0}$ in $\mathbb{Z}/p\mathbb{Z}$ satisfies $|R_{p}(g_{n,0})| \geq \text{min} \{ p, \frac{(2l-1)(p-1)}{(n,p-1)} + 1 \}$
\end{theorem}
See \cite{Melvyn}, page 60. 

\begin{lemma}
Let $ p > 3 $ be a prime number and  $l,n \in\mathbb{N}$ such that $1 < (n, p-1) < \frac{p-1}{2}$. Also let $d = (n, p-1)$ then we have R$_{p}$(g$_{n,0}$) = R$_{p}$(g$_{d,0}$), regardless of the number of variables l taken in g$_{n,0}$. 
\end{lemma}

\begin{proof} Lemma $4.1$ implies that the set of n-th powers and set of d-th powers in $\mathbb{Z}/p\mathbb{Z}$ are same for $d = (n,p-1)$. This means that the range of their sums are same the same i.e. R$_{p}$(g$_{n,0}$) = R$_{p}$(g$_{d,0}$). 
\end{proof}

\begin{lemma}
Let $ p > 3 $ be a prime number, $l,n  \in\mathbb{N}$ such that $1 < (n, p-1) < \frac{p-1}{2}$ and $d = (n, p-1)$ i.e. $p = (ds+1)$ for some  $s \in\mathbb{N}$ then $|R_{p}(g_{n,0})|$ $\geq$ min $\{$ $p, (2l-1)s+1$ $\}$, regardless of the number of variables l $\in\mathbb{N}$ taken in g$_{n,0}$. 
\end{lemma}

\begin{proof} Since we saw in Lemma 4.3 that R$_{p}$(g$_{n,0}$) = R$_{p}$(g$_{d,0}$), regardless of $l$, we could replace $l$ by $d$ in the statement of the lemma. Since $p = ds + 1$ we have $\frac{p-1}{d} = s$ which gives us $|$ R$_{p}$(g$_{n,0}$) $|$ $\geq$ min $\{$ $p, (2l-1)s+1 $ $\}$ on application of Theorem 4.2. 
\end{proof}

\begin{theorem}\textbf{[Generalized Cauchy-Davenport Theorem] }

Let $h \geq 2$ be a natural number, p be a prime number and $\emptyset\neq A \subset\mathbb{Z}/p\mathbb{Z}$. Define $ hA = \{ a_{1} + ... + a_{h} : a_{i}\in A$ \text{for every} $ 1 \leq i \leq h \}$. Then $|hA| \geq $ min $\{p, h|A|-h+1 \}$
\end{theorem}
[See \cite{Melvyn}, page 44]

\section{Intersectivity of $f_{n,k}$ for Odd $n$}

Throughout this section, $n$ will denote a odd natural number greater than $1$. If $p$ is a prime such that $ p \nmid n $  and $ k \in\mathbb{Z}$ , then $f_{n,k} \equiv $ 0 ( mod $p^{i}$) has solutions for every $ i \geq 1 $ if and only if the equation $f_{n,k} \equiv $ 0 ( mod $p$) has a non-zero solution. This is a consequence of Hensel's Lemma. Proving $f_{n,k} \equiv $ 0 ( mod $p$) has solutions for all $k \in\mathbb{Z}$ is equivalent to $R_{p}(f_{n,0}) = \mathbb{Z}/p\mathbb{Z}$ i.e. $|R_{p}(f_{n,0})| = p$. 

Since $f_{n,0} = lA_{n} = \{ a_{1} + a_{2} + ... + a_{l} : a_{1}, a_{2}, ... , a_{l} \in A_{n} \}$, where $l = \lceil \frac{n+1}{2} \rceil$ is the number of variables in $f_{n,0}$, it suffices to show that $|lA_{n}| = p$ for $ l = \lceil \frac{n+1}{2} \rceil$.  To summarize, whenever $p \nmid n$, $|lA_{n}| = p$ implies that $f_{n,k} \equiv 0 $ ( mod $p^{i})$ has solutions for every $ k,i \in\mathbb{N}$. Hence we have transformed the problem of showing existence of roots of the given polynomial modulo powers of such primes into a problem of showing that the cardinality of the $l$-fold sumset of $n$-th powers in $\mathbb{Z}/p\mathbb{Z}$ is all of $\mathbb{Z}/p\mathbb{Z}$. 

\begin{lemma}
Let  $n \in\mathbb{N}$, $k\in\mathbb{Z}$ and $p$ be an odd prime. Also let $ 1 > d = (n, p-1)$ i.e. $ p = (ds+1)$ for some $s\in\mathbb{N}$ and $s \geq 2$. Then we have the following:
\begin{itemize}

    \item For $h = d$, we have that $|hA_{n}| = p$.
    
    \item If $ p > (2d+1)$, then for $ h = \lceil \frac{d+1}{2} \rceil $, $|R_{p}(x_{1}^{n} + x_{2}^{n} + ... + x_{h}^{n})| = p $. 
    
\end{itemize}
\end{lemma}

\begin{proof}
Using Lemma 4.1 we have that $|A_{n}| = |A_{d}| = (s+1)$. If we have $h = d$, then $ h |A_{d}| - h + 1 =  h(s+1)-h+1 = (hs+1) = (ds+1) = p$. This along with Theorem $4.5$ gives $|hA_{d}| \geq$ min $\{p, h|A_{d}|-h+1\}$ = min $\{p, p\}$ = $p$. 

Now assume that $p > (2d+1)$ i.e. $1 < d < \frac{p-1}{2}$, then Lemma $4.4$ implies that for every $m \in\mathbb{N}$ we have $|R_{p}(x_{1}^{n} + x_{2}^{n} + ... + x_{m}^{n})| \geq$ min $\{ p, (2m-1)s+1\} $. Therefore, if we define $ h = \lceil \frac{d+1}{2} \rceil $, we have that $ (2h-1) \geq d $ implying that $ (2h-1)s + 1 \geq (ds+1) = p$. Therefore it follows that for $ h = \lceil \frac{d+1}{2} \rceil $, $|R_{p}(x_{1}^{n} + x_{2}^{n} + ... + x_{h}^{n})| = p $. 
\end{proof}

\begin{lemma}
Let $n \in\mathbb{N}$, $k\in\mathbb{Z}$ be such that $n > 1$ is odd and let $p$ be prime such that $ p > (2n+1) $. Then $f_{n,k} \equiv $ 0 ( mod $p^{i}$) has solutions $\forall i \geq 1$. 
\end{lemma}

\begin{proof}
Since $ p > (2n+1) $, $ p \nmid n$. Now assume that $(x_{1}, ... , x_{l})$ is a solution to $f_{n,k} \equiv $ 0 ( mod $p$) such that $p \mid x_{i}$ for every $1 \leq i \leq l$. Then we will have that $p \mid k$, and hence $(1, p-1, ... , 0 , ... , 0)$ would be a solution to $f_{n,k} \equiv $ 0 ( mod $p$). Therefore, without loss of generality we can assume that $p \nmid x_{1}$. 

So by Hensel's Lemma, existence of a solution to $f_{n,k} \equiv $ 0 ( mod p) will ensure the existence of a solution to $f_{n,k} \equiv $ 0 ( mod $p^{i}$) for every $ i \geq 1$.  In other words, it suffices to show that for a fixed  odd $n$ and any prime $ p > (2n+1) $, we have $|lA_{n}| = p$. 

Define $d = (n, p-1)$ and note that $d \leq n$ because $d \mid n$. If $d = 1$ then by Lemma $4.1$, we have that the set of $n$-th powers in $\mathbb{Z}/p\mathbb{Z}$ is all of $\mathbb{Z}/p\mathbb{Z}$ and the result follows. Therefore from now on, we assume that $ 1 < d \leq n$. 

Regardless of the value of $d$, $ 1 < d \leq n$, since $p > (2n+1)$, we have that $p > (2d+1)$, which means that we can apply the second itemized result from Lemma $5.1$. Therefore we have that for $h = \lceil \frac{d+1}{2} \rceil$, we have $|R_{p}(x_{1}^{n} + x_{2}^{n} + ... + x_{h}^{n})| = p $ i.e. $|hA_{n}| = p$. Since $d \leq n$ we have that $h = \lceil \frac{d+1}{2} \rceil \leq \lceil \frac{n+1}{2} \rceil$.  So, $|lA_{n}| = p$ for l = $\lceil \frac{n+1}{2} \rceil$ as desired.  
\end{proof}

\begin{lemma}
Let $n\in\mathbb{N}$ be an odd integer and $p$ be an odd prime. Then for $m = \frac{p-1}{2}$, we have $R_{p}(x_{1}^{n}+ ... + x_{m}^{n}) = \mathbb{Z}/p\mathbb{Z}$. 
\end{lemma}

\begin{proof}
Note that since $n$ and $p$ are odd, the set of $n$-th powers in $\mathbb{Z}/p\mathbb{Z}$ has at least 3-elements 0, 1 and -1. In addition, S = $\{-\frac{(p-1)}{2}, ... , -1, 0, 1, ... , \frac{p-1}{2} \}$ is a complete set of coset representatives modulo p. Therefore, if $ m = \frac{p-1}{2}$ we could write any coset representative in S, say $i \neq 0$ as a sum of $i$ $1$s or -$1$s depending on whether $i$ is positive or negative. This together with 0 = 1+(-1) implies $R_{p}(x_{1}^{n}+ ... + x_{m}^{n}) = \mathbb{Z}/p\mathbb{Z}$.  
\end{proof}

\begin{lemma}
Let  $n\in\mathbb{N}$, $k\in\mathbb{Z}$ such that $n \geq 3$ is odd and let $p$ be a prime such that $ p < (2n+1) $ and $p \nmid n$. Then $f_{n,k} \equiv $ 0 ( mod $p^{i}$) has solutions $\forall i \geq 1$. 
\end{lemma}

\begin{proof}
Similar to Lemma $5.2$, it suffices to show that $|lA_{n}|$ = $p$ for $l = \lceil \frac{n+1}{2} \rceil$ because $ p \nmid n $.  Define $d = (n, p-1)$ then we have that $p = ds + 1$. Since $p < (2n+1) $ , $d = n$ implies $p = (n+1)$. However, since $n$ is odd, $p = (n+1)$ cannot be prime. Therefore, we have that $d \leq \frac{n}{3}$ because $d \mid n $ and $d < n$. 

Recall that by Lemma 4.1, $g_{n,0} = g_{d,0}$. Since $d = (n, p-1)$ we could write $p = (ds+1)$. Using the first result from Lemma $5.1$, we have that for $h = d $, $|hA_{n}| = p $. Since $ h = d \leq \frac{n}{3} < l = \lceil \frac{n+1}{2} \rceil$, we have that $|lA_{n}| = p$ as desired. 
\end{proof}

\begin{theorem}\textbf{[Characterization of Intersectivity of $f_{n,k}$]}

Let $n > 1$ be an odd integer, $k \in\mathbb{Z}$ and $n = \prod_{i=1}^{t} p_{i}^{a_{i}}$ be the unique prime factorization of n. Define $N = \prod_{i=1}^{t} p_{i}^{a_{i}+1}$. Then $f_{n,k}$ is intersective if and only if following is satisfied:

\begin{enumerate}

    \item If $(2n+1)$ is not prime, then $f_{n,k} \equiv$ 0 ( mod N) is solvable. 
    
    \item If $(2n+1)$ is prime then $f_{n,k} \equiv$ 0 ( mod $(2n+1)N$) is solvable.
\end{enumerate} 

\end{theorem}

\begin{proof} Necessity of the conditions follows from the definition the of intersectivity of the polynomial $f_{n,k}$. Therefore, we will now prove the sufficiency. Using Lemmas $5.2$ and $5.4$, we already have that $f_{n,k} \equiv$ 0 ( mod $p^{i}$) has solutions for any $i \geq 1$ and all primes p such that $p \neq (2n+1)$ and $p \neq p_{i}$ for any $ 1 \leq i \leq n$. Furthermore, if $p = p_{i}$  then by Theorem $3.3$ it is enough to have a solution to $f_{n,k}$ $\equiv$ 0 ( mod $p_{i}^{a_{i}+1}$) to ensure that $f_{n,k} \equiv$ 0 ( mod $p_{i}^{j}$) has a solution for any $j \geq 1$.

Therefore if $(2n+1)$ is not prime, $f_{n,k}$ $\equiv$ 0 ( mod m) has solutions for any natural number $m$ as long as $f_{n,k}$ $\equiv$ 0 ( mod $p_{i}^{a_{i}+1}$) has solution for all $ 1 \leq i \leq n$, which is same as saying that $f_{n,k} \equiv$ 0 ( mod N) is solvable, by the Chinese Remainder Theorem. 

If $(2n+1)$ is prime, then for $f_{n,k}$ to be intersective, we also need existence of a solution to $f_{n,k} \equiv 0$ ( mod $(2n+1)$). Note that since $(2n+1)\nmid n$, Hensel's lemma will ensure existence of solutions to $f_{n,k} \equiv $ 0 , modulo higher powers of $(2n+1)$.  Therefore the existence of a solution to $f_{n,k} \equiv 0$ ( mod $(2n+1)N$) is enough to ensure that $f_{n,k} \equiv 0$ ( mod $m$) has solution for every $ m \geq 1$. 
\end{proof}

\section{Intersectivity of $f_{n,k}$ for Even $n$}

Recall that we used Lemmas $5.1 $ to $5.3$ to obtain the characterization of intersectivity of $f_{n,k}$ for odd $n$ in Theorem $5.5$. When $n$ was odd and $p$ a prime such that $p >(2n+1)$, without loss of generality, we could assume that the solution to $(x_{1}^{n} + ... + x_{l}^{n} - k ) \equiv$ 0 ( mod $p$) was non-zero. This was because both $-1$ and $1$ were always a $n^{th}$-power residues. Then, we could apply Hensel's Lemma. However when $n$ is even, $1$ is a $n^{th}$-power residue but $-1$ is not. Therefore the existence of a non-zero solution to $(x_{1}^{n} + ... + x_{l}^{n} - k ) \equiv$ 0 ( mod $p$) may not follow when $p \mid k$. 

Hence to obtain results akin to Lemmas $5.1$ to $5.3$ for even $n$, we have to replace the set $A_{n}$ by $A_{n}^{*}$ = $A_{n}\setminus\{0\}$. This will ensure that whenever we have $(x_{1}^{n} + ... + x_{l}^{n} - k ) \equiv$ 0 ( mod $p$) for even $n$ and prime $p$, there exists $i$ with $1 \leq i \leq l$ and $p \nmid x_{i}$. This means that we can use Hensel's Lemma whenever $p \nmid n$. In fact, given a fixed $k$, one only needs to use $A_{n}^{*}$ for those primes $p$ that divide $k$. However due to our interest in characterization of intersectivity of $f_{n,k}$ for any $k \in\mathbb{Z}$, replacing $A_{n}$ by $A_{n}^{*}$ for all primes $p$ is desirable. For exactly the same reason, we also want to replace $R_{p}(g_{n,0})$ by $R_{p}^{*}(g_{n,0})$. 

Recall that whenever $(n, p-1) = d$ and $p = (ds+1)$, we had $|A_{n}| = |A_{d}| = (s+1)$. Therefore, $|A_{n}^{*}| = |A_{d}^{*}| = s$. 

\begin{lemma}
Let $p > 3$ be a prime number, $l,n \in\mathbb{N}$ such that $n$ is even and $1 < (n, p-1) < \frac{p-1}{2}$. Also assume that $1 < d = (n, p-1) $ i.e. $p = (ds+1) $ for some $s\in\mathbb{N}$, then $|R_{p}^{*}(g_{n,0})| \geq $ min $\{p, (2l-1) s\}$ regardless of number of variables. 
\end{lemma}

\begin{proof}
Note that by Lemma $4.4$, that $|R_{p}(g_{d,0})| \geq $ min $\{p, (2l-1) s+1 \}$. If x $\in R_{p}(g_{d,0})$ such that $x$ is sum of $l$ non-zero elements of $\mathbb{Z}/p\mathbb{Z}$ then $x \in R_{p}^{*}(g_{d,0})$. This is because $R_{p}^{*}(g_{l,0}) \subset R_{p}^{*}(g_{d,0})$ for any $l \leq d$. Therefore, only element that possibly might not be in  $R_{p}^{*}(g_{d,0})$ is x = 0. Therefore, $|R_{p}^{*}(g_{d,0})| \geq |R_{p}(g_{d,0})| - 1  \geq $ min  $\{p, (2l-1) s \}$
\end{proof}

\begin{lemma}
Let $n, k \in\mathbb{N}$ such that $n$ is even and $p$ be an odd prime. Also assume that $ 1 > (n, p-1) = d$ i.e. $p = (ds+1)$ for some $s \in\mathbb{N}$. Then we have the following:
\begin{itemize}
    \item When $s\geq 2$, for $h = \lceil \frac{ds}{s-1} \rceil$ we have $|hA_{n}^{*}| = p$.
    
    \item When $p > (2d+1) $, for $h = \lceil \frac{d+2}{2} \rceil$, $|hA_{n}^{*}| = p$. 
\end{itemize}
\end{lemma}

\begin{proof}
Lemma 4.1 implies that $A_{n}^{*} = A_{d}^{*}$. Hence it suffices to show that when $s \geq 2$, $h = \lceil \frac{ds}{s-1} \rceil $ implies $|hA_{d}^{*}| = p$ to prove that $|hA_{n}^{*}| = p$. Assume that $h = \lceil \frac{ds}{s-1} \rceil $, then we have $hs - h \geq ds$ implying that $hs - h + 1 \geq p$ i.e. $h|A_{d}^{*}| - h + 1 \geq p$. This together with Theorem $4.5$, we have that $|hA_{d}^{*}| = p$.  

Now assume that $p > (2d+1) $ i.e. $ 1 < d < \frac{p-1}{2}$, hence we could apply Lemma $6.1$. Since $h = \lceil \frac{d+2}{2} \rceil$ we have that $h \geq \frac{d+1+\frac{1}{s}}{2}$ implying that $(2h-1)s \geq (ds+1) = p$. Therefore by Lemma $6.1$ we have that $|hA_{n}^{*}| = p$. 
\end{proof}

\begin{lemma}
Let $n\in\mathbb{N}$ such that $n\geq 2$ is even and $p$ be a prime such that $p > (2n+1) $ then $f_{n,k} \equiv $ 0 ( mod $p^{i}$) has solutions $\forall$ $k \in\mathbb{Z}$ and $\forall$ $i \geq 1$. 
\end{lemma}

\begin{proof}
Since $d = (n, p-1)$ we have $d \leq n$. Since $p > (2n + 1) \geq (2d+1) $, we always have that $p > (2d+1)$. Therefore, using the second itemized result from Lemma $6.2$ we have that for $h = \lceil \frac{d+2}{2} \rceil \leq \lceil \frac{n+2}{2} \rceil$, we have $|hA_{n}^{*}| = p$ i.e . $(x_{1}^{n} + ... + x_{h}^{n} - k ) \equiv 0 $ ( mod $p$) has solution. Since $h \leq l = $ max $\{ \lceil \frac{2n}{3} \rceil, \lceil \frac{n+2}{2} \rceil \}$, we have that $f_{n,k} \equiv 0$ ( mod $p$) has solution. By the comments made in the start of the section, we could apply Hensel's Lemma and conclude that $f_{n,k} \equiv 0$ ( mod $p^{i}$) for any $k \in\mathbb{Z}$ and any $i \geq 1 $. 
\end{proof}

\begin{lemma}
Let $n \in\mathbb{N}$ such that $n \geq 2$ be an even number and let $p$ be a prime such that $p < (2n+1) $. Also assume that $p \nmid n$ and $p \neq (n+1) $ then $f_{n,k} \equiv 0$ ( mod $p^{i}$) has solutions $\forall$ $k \in\mathbb{Z}$ and $\forall$ $i \geq 1$. 
\end{lemma}

\begin{proof}
Since $p < (2n+1) $ and  $p \neq (n+1) $ we have that $d = (n, p-1) \leq \frac{n}{2}$. Since $p \neq (n+1) $ , if $d = \frac{n}{2}$ then $p = (\frac{n}{2}+1)$. Therefore for $h = p = (\frac{n}{2}+1)$, we certainly have $|hA_{n}^{*}| = p$ implying that $f_{n,k} \equiv 0$ ( mod $p$) has solutions, which could be lifted by Hensel's Lemma because $p = (\frac{n}{2}+1) \nmid n$.

Therefore from now on, we assume that $d \leq \frac{n}{3}$. Therefore by first itemized result from Lemma $6.2$ we have that for $h = \lceil \frac{ds}{s-1} \rceil$, $|hA_{n}^{*}| = p$, whenever $p = ds + 1 $ for $s \geq 2$. However, we see that $h = \lceil \frac{ds}{s-1} \rceil \leq 2d \leq \frac{2n}{3} \leq l = $ max $\{ \lceil \frac{2n}{3} \rceil, \lceil \frac{n+2}{2} \rceil \}$, we have that $|lA_{n}^{*}| = p$ i.e. $f_{n,k} \equiv 0$ ( mod $p$) has solutions. As we mentioned at the start of this section, since we are dealing with $A_{n}^{*}$ , we could apply Hensel's Lemma and hence we have solution to $f_{p,k} \equiv 0$ ( mod $p^{i}$) for any $k \in\mathbb{Z}$ and any $i \geq 1$.  
\end{proof}

So far we have shown that $f_{n,k} \equiv 0$ ( mod $p^{i}$) has solutions for every $k\in\mathbb{Z}$ and for every $i \geq 1$, as long as the prime $p$ is such that $p \nmid n$ and $p \not\in \{(2n+1), (n+1) \}$. So, we only need solutions to $f_{n,k} \equiv 0$ ( mod $p^{i}$) when either $p$ divide $n$ or $p \in\{(2n+1), (n+1) \}$.

\begin{theorem}\textbf{[Characterization of Intersectivity of $f_{n,k}$]}\\
Let $n \geq 2$ be an even integer, $k \in\mathbb{Z}$ and $n = 2^{a} \prod_{i=1}^{t} p_{i}^{a_{i}}$ be the unique prime factorization of $n$. Assume that $k$ is not a sum of $l$ =  max $\{ \lceil \frac{2n}{3} \rceil , \lceil \frac{n+2}{2} \rceil \}$ many integer $n^{th}$ powers. Then $f_{n,k}$ is intersective if and only if following conditions are satisfied:

\begin{enumerate}
    \item If neither $(n+1)$ nor $(2n+1)$ is prime then $(a)$ and $(b)$ holds:

\begin{enumerate}

    \item For every $1 \leq i \leq t$, $f_{n,k} \equiv$ 0 ( mod $p_{i}^{j_{i}})$ has a solution $(x_{1}, ... , x_{l})$ for some $j_{i} \geq (a_{i}+1) $ such that $p \nmid x_{j}$ for some $ 1 \leq j \leq l$. 
    
    \item $f_{n,k} \equiv 0$ ( mod $2^{i}$) has solution $(x_{1}, ... , x_{l}) $ for some $i \geq (a+2) $ such that $2 \nmid x_{j}$ for some $1 \leq j \leq l$.
    
\end{enumerate}

\item If $p_{1}, p_{i} \in\{n+1, 2n+1 \}$ are primes for $ i = 1$ or $2$ then, along with the conditions $(a)$ and $(b)$ listed above, the following holds too:

\begin{itemize}
    
    \item $f_{n,k} \equiv 0$ ( mod $p_{j}$) has solution for every $j \in \{1, i\}$
\end{itemize}

\end{enumerate}
\end{theorem}

\begin{proof}
Note that $f_{n,k} \equiv 0$ ( mod $p^{i}$) has solution for any $k\in\mathbb{Z}$, any $i \geq 1$ and all primes $p\nmid n$ and $p \not\in\{(n+1), (2n+1)\}$. For primes $p \mid n$, the conditions listed in $(1)$ are sufficient and necessary for $f_{n,k} \equiv 0$ to be solvable modulo all powers of $p$ from Theorem $3.3$. So $(1)$ is necessary and sufficient condition for $f_{n,k}$ to be intersective, if $(n+1) $ and $(2n+1)$ are not prime.

If one or both of $(n+1) $ and $(2n+1)$ are prime then conditions listed in $(2)$ is necessary and sufficient for $f_{n,k}$ to be intersective. This follows from the earlier paragraph and from Hensel's Lemma, since $(n+1) \nmid n$ and $(2n+1) \nmid n$. 
\end{proof}

\section{Discussion and Examples}

For a given $n \geq 2$, $k \in\mathbb{Z}$ and $m = l, l+1, ... $ define  $ P_{m,k} = ( \sum_{i=1}^{m} x_{i}^{n} - k )$. Here $l = l(n)$ is as defined in the Section $2$, so $P_{l,k} = f_{n,k}$. Using the results we have proved, the intersectivity of $P_{m,k}$ can be summarized in the tables below for $n = 3, 4, 5, 6$ and $7$ respectively. 

For the sake of brevity, we will say that $P_{m,k} \equiv$ 0 ( mod $n$) is solvable \textit{nicely} if for every odd prime $p$ with $p^{a} \mid\mid n$ ( $a \geq 1$) , $P_{m,k} \equiv$ 0 ( mod $p^{a+1}$) has a solution $(x_{1}, ... , x_{m})$ such that $p \nmid x_{i}$ for some $ 1 \leq i \leq m$ and if $2^{a} \mid\mid n$ for some $a \geq 1$ then $P_{m,k} \equiv$ 0 ( mod $2^{a+2}$) has a solution $(x_{1}, ... , x_{m})$ such that $2 \nmid x_{i}$ for some $ 1 \leq i \leq m$. \vspace{2mm}

\begin{tabular}{ |p{3cm}|p{3cm}|p{5cm}| }
 \hline
 \multicolumn{3}{|c|}{n = 3} \\
 \hline
 \textbf{value of m}  & \textbf{$P_{m,k}$} & \textbf{ $P_{m,k}$ is intersective if and only if} \\
 \hline
 $l = 2$ & $(x_{1}^{3} + x_{2}^{3} - k )$ & $(x_{1}^{3} + x_{2}^{3} - k ) \equiv$ 0 ( mod $63$) is solvable \\ 
 \hline
 $3$ & ($\sum_{i=1}^{3} x_{i}^{3})  - k$ & ($\sum_{i=1}^{3} x_{i}^{3}$)  - k $\equiv$ 0 ( mod 9) is solvable\\
 \hline
 $4$ & ($\sum_{i=1}^{4} x_{i}^{3})  - k$ & Always intersective for every $ k \in\mathbb{Z}$\\
 \hline
\end{tabular}
\vspace{2mm}

The second line in the above table is due to the fact that the set cubes in $\mathbb{Z}/7\mathbb{Z}$ is $\{0, 1, 6 \}$. Similarly, the last line of the table is due to the fact that the set of cubes in $\mathbb{Z}/9\mathbb{Z}$ is $\{0, 1, 8\}$. The last line of the table states that every integer is sum of four cubes locally. This fact follows from the classical result by Davenport that almost every natural number is sum of four cubes \cite{Dav}. 

\vspace{2mm}

\begin{tabular}{ |p{3cm}|p{3cm}|p{5cm}| }
 \hline
 \multicolumn{3}{|c|}{n = 4} \\
 \hline
 \textbf{value of m}  & \textbf{$P_{m,k}$} & \textbf{$P_{m,k}$ is intersective if and only if} \\
 
 \hline
 
 $l = 3$ & ($\sum_{i=1}^{3} x_{i}^{4})  - k$ & ($\sum_{i=1}^{3} x_{i}^{4})  - k \equiv$ 0 ( mod $80)$ is solvable nicely \\ 
 \hline
 
 $4$ & ($\sum_{i=1}^{4} x_{i}^{4})  - k$ & ($\sum_{i=1}^{4} x_{i}^{4})  - k \equiv$ 0 ( mod $80$) is solvable nicely\\
 
 \hline
 
 $5$ to $15$ & ($\sum_{i=1}^{m} x_{i}^{4})  - k$ & ($\sum_{i=1}^{m} x_{i}^{4})  - k \equiv$ 0 ( mod $16$) is solvable nicely\\
 
 \hline
 
 $16$ & ($\sum_{i=1}^{16} x_{i}^{4})  - k$ & Always intersective for every $ k \in\mathbb{Z}$\\
 \hline
\end{tabular}

\vspace{2mm}

The third line in above table is due to the fact that the set of fourth powers in $\mathbb{Z}/5\mathbb{Z}$ is $\{0,1\}$; hence set of sums of five (or more) fourth powers is entire $\mathbb{Z}/5\mathbb{Z}$. The last line is due to the fact that the set of fourth powers in $\mathbb{Z}/16\mathbb{Z}$ is $\{0, 1\}$. 

\vspace{2mm}

\begin{tabular}{ |p{3cm}|p{3cm}|p{5cm}| }
 \hline
 \multicolumn{3}{|c|}{n = 5} \\
 \hline
 \textbf{value of m}  & \textbf{$P_{m,k}$} & \textbf{$P_{m,k}$ is intersective if and only if} \\
 
 \hline
 
 l = 3 & $(\sum_{i=1}^{3} x_{i}^{5}) - k $ & $(\sum_{i=1}^{3} x_{i}^{5}) - k \equiv $ 0  ( mod $11.25$) is solvable\\
 \hline
 l = 4 & $(\sum_{i=1}^{4} x_{i}^{5}) - k $ & $(\sum_{i=1}^{4} x_{i}^{5}) - k \equiv $ 0  ( mod $11$) is solvable\\
 \hline
 l = 5 & $(\sum_{i=1}^{5} x_{i}^{5}) - k $ & Always intersective for every $ k \in\mathbb{Z}$\\
 \hline
 \end{tabular}
 
 \vspace{2mm}
 
The set of fifth powers in $\mathbb{Z}/25\mathbb{Z}$ is $\{0, 1, 7, 18, 24\}$; hence the set of sums of at most four fifth powers is $\mathbb{Z}/25\mathbb{Z}$. This gives us the third line in the above table. In addition, the set of fifth power modulo $11$ is $\{0, 1, 10\}$ and hence the set of sums of at most five fifth powers is $\mathbb{Z}/11\mathbb{Z}$. This gives us the last line in the above table. 
\vspace{2mm}

\begin{tabular}{ |p{3cm}|p{3cm}|p{8cm}| }
 \hline
 \multicolumn{3}{|c|}{n = 6} \\
 \hline
 \textbf{value of m}  & \textbf{$P_{m,k}$} & \textbf{$P_{m,k}$ is intersective if and only if} \\
 
 \hline
 
 $l = 4$ & ($\sum_{i=1}^{4} x_{i}^{6})  - k$ & ($\sum_{i=1}^{3} x_{i}^{6})  - k \equiv$ 0 ( mod $(2^{3}.3^{2}.7.13)$) is solvable nicely \\ 
 \hline
 
 $5$ & ($\sum_{i=1}^{5} x_{i}^{6})  - k$ & ($\sum_{i=1}^{5} x_{i}^{6})  - k \equiv$ 0 ( mod $(2^{3}.3^{2}.7.13)$) is solvable nicely\\
 
 \hline
 
 $6$ & ($\sum_{i=1}^{6} x_{i}^{6})  - k$ & ($\sum_{i=1}^{6} x_{i}^{6})  - k \equiv$ 0 ( mod $(2^{3}.3^{2}.7)$) is solvable nicely\\
 
 \hline
 
 $7$ & ($\sum_{i=1}^{7} x_{i}^{6})  - k$ & ($\sum_{i=1}^{7} x_{i}^{6})  - k \equiv$ 0 ( mod $(2^{3}.3^{2})$) is solvable nicely\\
 
 \hline
 
 $8$ & ($\sum_{i=1}^{8} x_{i}^{6})  - k$ & ($\sum_{i=1}^{8} x_{i}^{6})  - k \equiv$ 0 ( mod $(3^{2})$) is solvable nicely\\
 
 \hline
 
 $9$ & ($\sum_{i=1}^{9} x_{i}^{6})  - k$ & Always intersective for every $ k \in\mathbb{Z}$\\
 \hline
\end{tabular}
\vspace{2mm}

The third line of the above table is due to the fact that the set of sixth powers in $\mathbb{Z}/13\mathbb{Z}$ is $\{0, 1, 12\}$; hence the set of sums of at most six sixth powers is entire $\mathbb{Z}/13\mathbb{Z}$. The fourth line of the table is simply because the set of sixth powers in $\mathbb{Z}/7\mathbb{Z}$ contains $1$. The fifth and sixth line is also due to exactly same reason in $\mathbb{Z}/8\mathbb{Z}$ and $\mathbb{Z}/9\mathbb{Z}$.
\vspace{2mm}

\begin{tabular}{ |p{3cm}|p{3cm}|p{5cm}| }
 \hline
 \multicolumn{3}{|c|}{n = 7} \\
 \hline
 \textbf{value of m}  & \textbf{$P_{m,k}$} & \textbf{$P_{m,k}$ is intersective if and only if} \\
 
 \hline
 
 $l = 4$ & $(\sum_{i=1}^{4} x_{i}^{7}) - k $ &  Always intersective for every $ k \in\mathbb{Z}$\\
 
 \hline

\end{tabular}
\vspace{2mm}

The third column in the above table is because seventh powers modulo $49$ are $0, 1, 2, 4, 9, 11, 15, 16,$ $18, 22, 23, 25, 29, 30, 32, 36, 37, 39, 43, 44$ and $46$. 
Therefore the set of sum of four seventh powers in $\mathbb{Z}/49\mathbb{Z}$ is entire $\mathbb{Z}/49\mathbb{Z}$. This also means any integer is sum of four seventh powers locally.\vspace{2mm} 

\textbf{Acknowledgment:} This research was partly supported by the NSF, under grant DMS-1812028.

\singlespacing
\bibliographystyle{plain}
\bibliography{intersective}

\begin{thebibliography}{10}

\bibitem{BerBil}
D.~Berend and Y.~Bilu.
\newblock {Polynomials with Roots Modulo Every Integer.}
\newblock {\em Proc. Amer. Math. Soc.}, 124(6):1663--1671, June 1996.

\bibitem{BLL}
V.~Bergleson, A.~Leibman, and E.~Lesigne.
\newblock {Intersective Polynomials and Polynomial Szemeredi Theorem}.
\newblock {\em Adv. Math.}, 219:369--388, 2008.
\newblock doi: \url{https://doi-org/10.1016/j.aim.2008.05.008}.

\bibitem{Dav}
H.~Davenport.
\newblock {On Waring's Problem for Cubes}.
\newblock {\em Acta Math.}, 71:123--143, 1939.

\bibitem{Fu}
H.~Furstenberg.
\newblock {Ergodic Behavior of Diagonal Measures and a Theorem of Szemer\'edi
  on Arithmetic Progressions}.
\newblock {\em J. d'Analyse Math}, 71:204--256, 1977.

\bibitem{HLS}
A.M. Hyde, D.P. Lee, and K.B. Spearman.
\newblock {Polynomials (x$^3$-n)(x$^2$+3) Solvable Modulo Any Integer}.
\newblock {\em Amer. Math. Monthly}, 121(4):355--358, April 2014.
\newblock doi: \url{https://doi.org/10.4169/amer.math.monthly.121.04.355}.

\bibitem{KaMF}
T.~Kamae and M.~Mend{\'e}s-France.
\newblock {Van der Corput's Difference Theorem}.
\newblock {\em Israel J. Math.}, 31:335--342, 1978.

\bibitem{Melvyn}
Melvyin~B. Nathanson~(1996).
\newblock {\em {Additive Number Theory - Inverse Problems and Geometry of
  Sumsets}}.
\newblock New York:Springer-Verlag.

\bibitem{Rosen}
Kenneth~H. Rosen~(2000).
\newblock {\em {Elementary Number Theory and its Applications}}.
\newblock Reading, Massachusetts: Addison Wesley Longman.

\bibitem{Sa1}
A.~S{\'a}rk{\"o}zy.
\newblock {On Difference Sets of Sequences of Integers}.
\newblock {\em Acta Math. Acad. Sci. Hungar.}, 31:125--149, 1978.

\bibitem{Sa2}
A.~S{\'a}rk{\"o}zy.
\newblock {On Difference Sets of Sequences of Integers}.
\newblock {\em Acta Math. Acad. Sci. Hungar.}, 31:355--386, 1978.

\bibitem{Vaughan}
R.C. Vaughan and T.D. Wooley~(2002).
\newblock Waring's problem : A survey.
\newblock In A.K. Peters, editor, {\em {Number theory for the millennium,
  III}}, pages 301--340.

\end{thebibliography}

\end{document}